\documentclass[a4paper,11pt]{article}
\usepackage{amsmath,amsthm,amssymb,graphicx,subfigure}
\usepackage{bbm}
\usepackage{graphicx,verbatim}
\usepackage{color}
\usepackage [english]{babel}
\usepackage [latin1] {inputenc}
\usepackage [T1] {fontenc}
\usepackage{amssymb}
\usepackage{amsfonts}
\usepackage{dsfont}
\usepackage{amsmath}
\usepackage{stmaryrd}
\usepackage{graphicx}
\newtheorem{theorem}{Theorem}[section]

\newtheorem{proposition}[theorem]{Proposition}

\newtheorem{remark}[theorem]{Remark}

\date{}
\title{Eigenvector convergence for minors of unitarily invariant infinite random matrices}
\author{
        Joseph \textsc{Najnudel}     \footnote{\texttt{joseph.najnudel@uc.edu}}
       } 

\begin{document}
\maketitle
\begin{abstract} 
In \cite{Pic91}, Pickrell fully characterizes  the unitarily invariant probability measures on infinite Hermitian matrices. An alternative proof of this classification is given by Olshanski and Vershik in 
 \cite{OV96}, and in \cite{BO01}, Borodin and Olshanski  deduce from this proof that under any of these invariant measures, the extreme eigenvalues of the minors, divided by the dimension, converge almost surely. In this paper, we prove that one also has a weak convergence for the eigenvectors, in a sense which is made precise. 
After mapping Hermitian to unitary matrices via the Cayley transform, our result extends a convergence proven in our paper with Maples and Nikeghbali \cite{MNN18}, for which a coupling of the Circular Unitary Ensemble of all dimensions is considered. 
\end{abstract} 
\section*{Introduction}
Let $\mathcal{H}$ be the set of infinite Hermitian matrices, i.e. infinite families $(m_{j,k})_{j,k \geq 1}$ of complex numbers 
such that $m_{j,k} = \overline{m_{k,j}}$, and $\mathcal{U}$ the group of infinite unitary matrices, i.e. matrices
$(u_{j,k})_{j,k \geq 1}$ such that there exists $n \geq 1$ satisfying the following property: $(u_{j,k})_{1 \leq j,k \leq n}$ is a unitary matrix and $u_{j,k} = \delta_{j,k} := \mathds{1}_{j = k}$ if $j$ or $k$ is strictly larger than $n$.  The group $\mathcal{U}$ can be considered as the union of $(U(n))_{n  \geq 1}$, where $U(n)$ is naturally embedded in $U(n+1)$
by the map $u \mapsto \operatorname{Diag} (u,1)$. 
The group $\mathcal{U}$ naturally acts on $\mathcal{H}$ by conjugation, and some probability measures on $\mathcal{H}$ are invariant by this action: they are called {\it central measures}. After a similar study, by Aldous \cite{A81}, of infinite random matrices which are invariant by left and right multiplication by permutation or orthogonal matrices, the central measures on $\mathcal{H}$ have been completely classified 
by Pickrell \cite{Pic91}, by Olshanski and Vershik in \cite{OV96}, and can be decomposed as convex combinations of measures called {\it ergodic measures}. The ergodic measures are indexed by the set $\mathbb{R} \times \mathbb{R}_+ \times \mathcal{S}$, where $\mathcal{S}$ contains all  square-summable sets of non-zero real numbers with possible repetitions. 
Moreover, in \cite{BO01}, Borodin and Olshanki show that these points correspond to almost sure limits of the extremal eigenvalues of the minors of the corresponding  infinite matrix, divided by their dimension. 
In gerenal, a central measure is then represented by a probability distribution on  $\mathbb{R} \times \mathbb{R}_+ \times \mathcal{S}$. This distribution has been studied in detail for some particular central measures, which enjoy remarkable properties. For example, Borodin and Olshanski \cite{BO01} have studied the case of the {\it Hua-Pickrell measures}, which depend on a complex parameter $s$ with real part strictly larger than $-1/2$, and under which the 
distribution of the minor $M_n := (m_{j,k})_{1 \leq j, k \leq n}$ has a density proportional to 
$\operatorname{det} (1 + iM_n)^{-s-n}  (1 - iM_n)^{- \bar{s} -n} )$ with respect to the Lebesgue measure on Hermitian matrices. Borodin and Olshanki proved that under the probability measure on $\mathbb{R} \times \mathbb{R}_+ \times \mathcal{S}$ associated to the Hua-Pickrell measure of parameter $s$, the third component (in $\mathcal{S}$) is a determinantal point process whose kernel, depending on $s$,  is explicitly computed. For $s = 0$, we get the inverses of the points of a determinantal sine-kernel process. 
The authors also show that the second component (called {\it Gaussian component}) is equal to zero for $s = 0$, and Qiu \cite{Qiu17} shows that it is the case for all $s$. He also determines the first component for $s \in \mathbb{R}$. 

The  case $s = 0$, which corresponds to minors following the Cauchy Ensemble, is particularly interesting for the following reason. The Cayley transform $x \mapsto (x - i)/(x + i)$ maps the Hermitian matrices to the unitary matrices for which $1$ is not an eigenvalue. The sequences of minors of infinite Hermitian matrices are mapped to some particular sequences of unitary matrices of increasing dimensions, called {\it virtual isometries}, and characterized by Neretin in \cite{Neretin02}. 
These virtual isometries defined by Neretin correspond to unitary matrices for which $1$ is not an eigenvalue: this constraint has been removed in a construction done in a joint paper with Bourgade and Nikeghbali \cite{BNN13}, which generalizes 
the construction of Neretin. Our notion of virtual isometry also generalizes the notation of {\it virtual permutations} introduced by Kerov, Olshanski and Vershik in \cite{KOV93} and studied in detail by Tsilevich \cite{Tsilevich98}, who gives a classification of the central measures on virtual permutations which is quite similar to the classification given by Olshanski and Vershik in the Hermitian setting. 
If we map an infinite Hermitian matrix following the Hua-Pickrell measure for $s = 0$ by the Cayley transform, we get a virtual isometry such that each component follows the Circular Unitary Ensemble, i.e. the Haar measure on the unitary group. The convergence results proven in \cite{OV96} and \cite{BO01} imply the following: if for a virtual isometry $(u_n)_{n \geq 1}$ following the Haar measure, we consider the eigenangles of $u_n$, multiplied by $n/2 \pi$, then the corresponding point measure a.s. converges locally weakly to a determinantal sine-kernel process. In \cite{BNN13}, we give an alternative and more direct proof of this result, with an estimate of the speed of convergence. 

In \cite{MNN13} and \cite{MNN18}, we improve this estimate, and we also show that each fixed component of the eigenvectors, multiplied by $\sqrt{n}$, almost surely converges to a non-trivial limit when $n$ goes to infinity. In \cite{MNN13}, we construct an operator $H$ on an infinite dimensional space, whose eigenvalues and eigenvectors are the limits of the renormalized eigenangles and eigenvectors of $(u_n)_{n \geq 1}$. 
The space $\mathcal{E}$ where the operator $H$ is defined is spanned by independent infinite sequences of complex i.i.d. Gaussian variables, and its structure is not classical: in particular, it is not a Hilbert space. The flow $(e^{iH \alpha})_{\alpha \in \mathbb{R}}$ of operators on $\mathcal{E}$ can then be seen as a limit, in a sense to be made precise, of 
the family $(u_n^{\lfloor \alpha n \rfloor})_{n \geq 1}$ of unitary matrices, when $n$ goes to infinity. 

The construction in \cite{MNN13} is, to our knowledge, the first natural construction of an operator whose spectrum is a determinantal sine-kernel process, and which is related to a classical ensemble of random matrices. A different construction of such an operator has been later given by Valk\'o and Vir\'ag in \cite{VV17}. Note that it is natural to expect that the sine-kernel process is the spectrum of some kind of universal random operator, since it appears as a limit for the spectrum of many matrix ensembles: however, until now, our attempts to construct an operator which is more universal  (i.e. related to many  random matrix ensembles)  than those given in \cite{MNN13} and \cite{VV17} have not succeeded (the operator in \cite{MNN13} is only related to the Circular Unitary Ensemble, and the operator in \cite{VV17} is related to ensembles of tridiagonal matrices). The construction of operators whose spectrum is the sine-kernel process might also be, even if this is very speculative, related to the conjecture of Hilbert and P\'olya, who suggested that the non-trivial zeros of the Riemann zeta functions should be interpreted as the spectrum of an operator $\frac12 + iH$ with $H$ a Hermitian operator. Indeed, the zeros of $\zeta$ are believed to locally behave like a determinantal sine-kernel process, as deduced from a conjecture by Montgomery \cite{Montgomery73}, generalized by Rudnick and Sarnak \cite{RS96}. More detail on this discussion can be found in \cite{MNN13}. 

The main goal of the present paper is the generalization of the result of convergence of eigenvectors given in \cite{MNN13} and \cite{MNN18}: we will show that this convergence occurs for any random infinite Hermitian matrix following a central measure, or equivalently (by using the Cayley transform which preserves the eigenvectors), for any random virtual isometry in the sense of \cite{Neretin02}, whose distribution is invariant by unitary conjugation.

\section{Classification of the central measures and statement of our main result}
In this section, we first recall the classification of the central measures on infinite Hermitian matrices given by Pickrell \cite{Pic91}, and by Olshanski and Vershik \cite{OV96}.  A reformulation of this classification is given by the following proposition:
\begin{proposition} \label{propositionOV96}
Let $\mathbb{P}$ be a central probability measure on the space of infinite Hermitian matrices. Then, 
there exists a probability measure $\mu$ on $\mathbb{R} \times \mathbb{R}_+ \times \mathcal{S}$, such that 
$$\mathbb{P} = \int_{\mathbb{R} \times \mathbb{R}_+ \times \mathcal{S}} \mathbb{P}^{(\alpha)} d \mu(  \alpha),$$
where  $\mathbb{P}^{(\alpha)}$ is defined as follows.
For $\gamma_1 \in \mathbb{R}$, $\gamma_2 \in \mathbb{R}_+$, and a square-summable, finite or infinite, sequence $(x_\ell)_{\ell \geq 1}$ of non-zero real numbers,
let the infinite matrix $M = (m_{j,k})_{j,k \geq 1}$ be given by 
$$m_{j,k} = \gamma_1 \delta_{j,k}  + \sqrt{\gamma_2} G_{j,k} + 
\sum_{\ell \geq 1} x_{\ell} ( \xi^{(\ell)}_{j}  \overline{\xi^{(\ell)}_{k}} - \delta_{j,k}),$$
where $(G_{j,k})_{j,k \geq 1}$ is an infinite matrix following the Gaussian Unitary Ensemble (normalized in order to have 
$\mathbb{E} [ G_{j,j}^2 ] = 1$) and
 $(\xi^{(\ell)}_{j}  )_{\ell, j \geq 1}$ is an independent family of i.i.d. complex Gaussian variables, such that 
$\mathbb{E} [ \xi^{(\ell)}_{j}] = \mathbb{E} [ (\xi^{(\ell)}_{j})^2] = 0$, 
$\mathbb{E} [ |\xi^{(\ell)}_{j}|^2] = 1$. Then, $M$ follows the distribution $\mathbb{P}^{(\alpha)}$ where 
$$\alpha = (\gamma_1, \gamma_2, \{x_{\ell}, \ell \geq 1\}).$$
\end{proposition}
\begin{remark}
The space $\mathcal{S}$ is endowed by the $\sigma$-algebra generated by the topology of weak convergence of point measures on compact sets of $\mathbb{R}^*$. If the sequence $(x_{\ell})_{\ell \geq 1}$ is infinite, then 
the infinite sum defining $m_{j,k}$ is convergent almost surely and in $L^2$, since the partial sums form a martingale which is bounded in $L^2$ (because of the square-summability of $(x_{\ell})_{\ell \geq 1}$). 
\end{remark}
In \cite{OV96} and \cite{BO01}, Borodin, Olshanski and Vershik show that under 
$\mathbb{P}^{(\alpha)}$ with $\alpha$ given just above, the extremal eigenvalues of the minors of $M$, divided by the dimension of the minors, tend to the points $(x_{\ell})_{\ell \geq 1}$. 
We will show a similar result of convergence for the coordinates of the eigenvectors. For example, if $x_{\ell}$ is the unique $\ell$-th largest point of the sequence $(x_{\ell})_{\ell \geq 1}$ and if $x_{\ell} > 0$, then each component of the eigenvector associated to the $\ell$-th largest eigenvalue, properly renormalized, converges to the corresponding component of the infinite sequence $(\xi^{(\ell)}_k)_{k \geq 1}$.
In order to state the result in full generality, it is not very convenient to directly consider eigenvectors, for the following reasons: 
\begin{itemize}
\item The normalization of the eigenvector depends on the arbitrary choice of a phase. 
\item If the sequence $(x_{\ell})_{\ell \geq 1}$ contains multiple points, then several eigenvalues tend to the same limit after dividing by the dimension of the minors, and the convergence of the individual eigenvectors is no longer true in general.  
\end{itemize}
A good way to avoid the first problem is to replace the eigenvectors by the corresponding matrices of  orthogonal projections, which are uniquely determined. The order of magnitude of the entries of these projections is $1/n$ (the trace is equal to $1$ if the eigenvalues are simple): hence, in order to get a possible convergence, it is natural to multiplies these entries by $n$. In order to solve the problem of "fusion of the eigenspaces" at the limit when $(x_{\ell})_{\ell \geq 1}$ has multiple points, we will consider the spectral projection-valued measures defined just below, instead of the individual 
 projections on eigenspaces.
\begin{theorem} \label{convergenceeigenvectors}
Let $\mathbb{P}$ be a central probability measure on the space of infinite Hermitian matrices, and let $M = (m_{j,k})_{j,k \geq 1}$ be an infinite matrix following the distribution $\mathbb{P}$.
Then, the random  measure
$$\Lambda_n :=  \sum_{\lambda \in \operatorname{Spec}(M_n)} m(\lambda) \delta_{\lambda/n}$$
where $M_n$ is the top-left $n \times n$ minor of $M$,  $m(\lambda)$ the multiplicity of the eigenvalue $\lambda$, and $\delta_{\lambda/n}$ the Dirac  measure at $\lambda/n$,
 converges almost surely  to a limiting atomic measure $\Lambda_{\infty}$, with finitely many atoms on $\mathbb{R} \backslash (-\epsilon, \epsilon)$ for all $\epsilon > 0$, in the following sense: for all intervals $I$ included in $\mathbb{R}_+$ or $\mathbb{R}_-$, whose boundary does not contain zero or a point of the support of $\Lambda_{\infty}$,
$$ \Lambda_{n} (I) \underset{n \rightarrow \infty}{\longrightarrow} \Lambda_{\infty} (I).$$
Moreover, for $a, b \geq 1$ and $n \geq \max(a,b)$, 
if we define the random complex measure 
$$\Sigma_{n,a,b} := \sum_{\lambda \in \operatorname{Spec}(M_n)} n ( \Pi_{M_n, \lambda})_{a,b} \, \delta_{\lambda/n},$$
where $(\Pi_{M_n, \lambda})_{a,b}$ is the $a,b$ entry of the matrix of the orthogonal projection on the eigenspace of $M_n$ associated to the eigenvalue $\lambda$, then there a.s. exists an atomic non-zero complex  measure $\Sigma_{\infty,a,b}$, with finitely many atoms on $\mathbb{R} \backslash (-\epsilon, \epsilon)$ for all $\epsilon > 0$, such that $\Sigma_{n,a,b}$ converges to $\Sigma_{\infty,a,b}$ in the  same sense as before. 
Moreover, if, with the notation of Proposition \ref{propositionOV96}, 
\begin{equation} \label{formulaOV}
m_{j,k} = \gamma_1 \delta_{j,k}  + \sqrt{\gamma_2} G_{j,k} + 
\sum_{\ell \geq 1} x_{\ell} ( \xi^{(\ell)}_{j}  \overline{\xi^{(\ell)}_{k}} - \delta_{j,k}),
\end{equation}
then we have
$$\Lambda_{\infty} = \sum_{\ell \geq 1} \delta_{x_{\ell}}$$
and
$$\Sigma_{\infty,a,b} = \sum_{\ell \geq 1}  \xi^{(\ell)}_{a}  \overline{\xi^{(\ell)}_{b}}  \delta_{x_{\ell}}$$
\end{theorem}
\begin{remark}
It is clear, from Proposition \ref{propositionOV96}, that it is enough to show the theorem when \eqref{formulaOV} holds.
The part of the theorem concerning the convergence of $\Lambda_n$ is in fact already proven in \cite{OV96} and \cite{BO01}. For sake of completeness, we give an alternative proof, together with the convergence of 
$\Sigma_{n,a,b}$. 
Note that the measure $\Sigma_{\infty,a,b}$ is not well-defined on intervals containing zero if the sequence $(x_{\ell})_{\ell \geq 1}$ is infinite. 
\end{remark}
We quite easily deduce the following result from Theorem \ref{convergenceeigenvectors}, which gives the convergence of renormalized eigenvalues and eigenvectors: 
\begin{proposition}
Let us assume that \eqref{formulaOV} occurs. 
For all $r \geq 1$, the $r$-th largest (resp. smallest) eigenvalue of $M_n$ (counted with multiplicity), divided by $n$, converges a.s. to the $r$-th largest (resp. smallest) point of $\{ x_{\ell}, \ell \geq 1\}$ (counted with multiplicity), if this point is positive (resp. negative), and to zero if this point is negative (resp. positive) or does  not exist. 

Let us now assume that $(x_{\ell})_{\ell \geq 1}$ has  a single $r$-th largest (resp. smallest) point  $x_{\ell(r)}$, and that this point is positive (resp. negative). Let $V_n$ be an eigenvector corresponding to the $r$-th largest (resp. smallest) eigenvalue of $M_n$, normalized in such a way that $||V_n|| = \sqrt{n}$ and the first non-zero coordinate of $V_n$ is real and positive. Then, for all $a \geq 1$, the $a$-th coordinate of $V_n$ converges a.s. to 
$ \xi^{(\ell(r))}_{a}  ( |\xi^{(\ell(r))}_{1}| /  \xi^{(\ell(r))}_{1})$ when $n$ goes to infinity. 
\end{proposition}
\begin{proof}
Let us assume that the $r$-th largest point of  $(x_{\ell})_{\ell \geq 1}$ is $y > 0$. 
If $z> y > 0$ is sufficiently close to $y$, $z$ is not in the sequence $(x_{\ell})_{\ell \geq 1}$ and 
$\Lambda_{\infty}( [z, \infty)) \leq r -1 $. Hence, by the convergence of $\Lambda_n$, 
$\Lambda_n  ([z, \infty)) < r$ for $n$ large enough and the $r$-th largest eigenvalue of $M_n$ is smaller than 
$n z$. Similarly, if $z \in (0,y)$ is sufficiently close to $y$, $z$ is not in
 $(x_{\ell})_{\ell \geq 1}$ and 
$\Lambda_{\infty}( [z, \infty)) \geq  r  $, which implies $\Lambda_n  ([z, \infty)) > r-1$ for $n$ large enough: the $r$-th largest eigenvalue 
is larger than or equal to $nz$. 

If the $r$-th largest point is negative or does not exist, there is no accumulation of points at the right of $0$, so for $\epsilon > 0$ small enough, $\epsilon$ is not in $(x_{\ell})_{\ell \geq 1}$, and $\Lambda_{\infty}  ([\epsilon, \infty)) \leq  r-1$, 
$\Lambda_n  ([\epsilon, \infty)) < r$ for large $n$, which implies that the $r$-th largest eigenvalue is smaller than $\epsilon n$. On the other hand, there exists $\epsilon > 0$ arbitrarily small such that $-\epsilon$ is not in 
$(x_{\ell})_{\ell \geq 1}$. The quantity $s := \Lambda_{\infty}  ((-\infty, -\epsilon]) $ is finite, 
and for $n$ large, $\Lambda_{n}  ((-\infty, - \epsilon]) < s+1$, which shows that the $(s+1)$-th smallest eigenvalue 
is larger than $- \epsilon n$. A fortiori, it is also the case of the $r$-th largest eigenvalue if $n$ is large. 

We have now proven the convergence of the largest eigenvalues, the proof for the smallest eigenvalue is exactly similar. 

Let us now consider the eigenvectors. Let us assume that  $x_{\ell(r)} > 0$ is the single $r$-th largest point of 
$(x_{\ell})_{\ell \geq 1}$. From the convergence of the eigenvalues, we know that for $0 < z< x_{\ell(r)} < t$, $z$ and $t$ being sufficiently close to $x_{\ell(r)}$, and 
for $n$ large enough depending on $z$ and $t$, the $r$-th largest eigenvalue of $M_n$  is simple and in  $(nz,nt)$, and it is the only eigenvalue in this interval (because the $(r-1)$-th eigenvalue is larger than $nt$ if $r \geq 2$,  and the $(r+1)$-th is smaller than $nz$). We deduce that if $V_n$ is an eigenvector associated to the $r$-th largest eigenvalue of $M_n$,  normalized as in the proposition, 
$$\Sigma_{n,a,b} ((z,t)) = (V_n)_a \overline{(V_n)}_b \underset{n \rightarrow \infty}{\longrightarrow}
\Sigma_{\infty,a,b} ((z,t)) =  \xi^{(\ell(r))}_{a}
\overline{ \xi^{(\ell(r))}_{b} } $$
Taking $ a = b = 1$, we get 
$$|(V_n)_1|^2 \underset{n \rightarrow \infty}{\longrightarrow}  |\xi^{(\ell(r))}_{1}|^2,$$
and then, from the normalization chosen for the phase, 
$$(V_n)_1 \underset{n \rightarrow \infty}{\longrightarrow}  |\xi^{(\ell(r))}_{1}|.$$
Taking $b = 1$ and general $a$, we have 
$$ (V_n)_a \overline{(V_n)}_1 \underset{n \rightarrow \infty}{\longrightarrow}
\xi^{(\ell(r))}_{a} \overline{\xi^{(\ell(r))}_{1}},$$
and then, dividing by the convergence above (the limit being a.s. different from zero), 
$$ (V_n)_a \underset{n \rightarrow \infty}{\longrightarrow}
\xi^{(\ell(r))}_{a} \overline{\xi^{(\ell(r))}_{1}} /  |\xi^{(\ell(r))}_{1}|
=  \xi^{(\ell(r))}_{a}  ( |\xi^{(\ell(r))}_{1}| /  \xi^{(\ell(r))}_{1}).$$

\end{proof}

We will prove Theorem \ref{convergenceeigenvectors} into two steps: we first consider the case where $\gamma_2 = 0$ and the sequence $(x_{\ell})_{\ell \geq 1}$ is finite, and then we deduce the general case. 
\section{Proof of Theorem \ref{convergenceeigenvectors} for finite sequences} 
If $\gamma_2 = 0$ and $(x_\ell)_{\ell \geq 1}$ has a finite number $p$ of elements, then we can write 
$$m_{j,k} = \gamma \delta_{j,k} + \sum_{\ell = 1}^{p} x_{\ell} \xi_j^{(\ell)} \overline{\xi_k^{(\ell)}}.$$
where 
$$\gamma = \gamma_1 - \sum_{\ell = 1}^{p} x_{\ell}.$$
Since the parameter $\gamma$ does not change the eigenvectors of the minors $M_n$ and shifts the eigenvalues by 
$\gamma$, it translates the measures $\Lambda_n$ and  $\Pi_{n,a,b}$ by $\gamma/n$, and then does not change their limiting measures.  

We can then assume 
$$m_{j,k} = \sum_{\ell = 1}^{p} x_{\ell} \xi_j^{(\ell)} \overline{\xi_k^{(\ell)}}.$$

For $n >  p$, the vectors $\xi_{[n]}^{(\ell)} := (\xi_j^{(\ell)} )_{1 \leq j \leq n} $, $ 1 \leq \ell \leq p$ are almost surely linearly independent. 
For any vector $V$ in $\mathbb{C}^n$, 
$$(M_n V)_{j} = \sum_{\ell = 1}^{p} x_{\ell}  \xi_j^{(\ell)}   \sum_{k=1}^n \overline{\xi_k^{(\ell)}} V_k,$$
i.e. 
$$M_n V =  \sum_{\ell = 1}^{p} x_{\ell}  \langle \xi_{[n]}^{(\ell)}, V \rangle  \, \xi_{[n]}^{(\ell)}.$$
We deduce that $E_n := \operatorname{Span} (  \xi_{[n]}^{(\ell)}, 1 \leq \ell \leq p )$, and its orthogonal, are invariant spaces for $M_n$, and that the orthogonal of $E_n$ is in the kernel of $M_n$. 
Hence, we have 
$$\Lambda_n  = (n-p) \delta_0 + \sum_{\lambda \in \operatorname{Spec} (P_n)} m(\lambda) \delta_{\lambda/n},$$
where $P_n$ is the restriction of $M_n$ to $E_n$, 
and 
$$\Sigma_{n,a,b} = n (\Pi_{(E_n)^{\perp}} )_{a,b}\delta_0 + 
 \sum_{\lambda \in \operatorname{Spec} (P_n)} n (\Pi_{P_n, \lambda})_{a,b} \delta_{\lambda/n},$$
where $\Pi_{(E_n)^{\perp}}$ is the orthogonal projection on $(E_n)^{\perp}$ and 
$\Pi_{P_n, \lambda}$ is the orthogonal projection on the eigenspace of $P_n$ corresponding to the eigenvalue $\lambda$. 
Since the convergence of measures defined in Theorem \ref{convergenceeigenvectors} does not involve Dirac masses at zero, it is enough to show that almost surely, 
\begin{equation}\sum_{\lambda \in \operatorname{Spec} (P_n) \cap nI} m(\lambda)
\underset{n \rightarrow \infty}{\longrightarrow} \sum_{\ell = 1}^p \mathds{1}_{x_{\ell} \in I}, \label{convergencesp1}
\end{equation}
and 
\begin{equation}\sum_{\lambda \in \operatorname{Spec} (P_n) \cap nI} n (\Pi_{P_n, \lambda})_{a,b} \, 
\underset{n \rightarrow \infty}{\longrightarrow} \sum_{\ell = 1}^p  \xi^{(\ell)}_{a}  \overline{\xi^{(\ell)}_{b}}  \mathds{1}_{x_{\ell} \in I}, \label{convergencesp2}
\end{equation}
for all intervals $I$ whose boundary does not contain a point in $(x_{\ell})_{\ell \geq 1}$. 
In the first convergence, if we divide by $p$, we simply get a classical convergence of probability measures. 
Taking the Fourier transform, it is then enough to show
\begin{equation}
\operatorname{Tr} (e^{i \mu P_n/n} ) \underset{n \rightarrow \infty}{\longrightarrow}
 \sum_{\ell = 1}^p e^{i \mu x_{\ell}}, \label{convergencesp123}
 \end{equation}
for all $\mu \in \mathbb{R}$. 
Now, in the basis $( \xi_{[n]}^{(\ell)})_{1 \leq \ell \leq p}$ of $E_n$, the operator $P_n/n$ has matrix
$$\left( \frac{1}{n} \sum_{\ell = 1}^p x_{\ell} \langle \xi_{[n]}^{(\ell)} ,  \xi_{[n]}^{(m)} \rangle \right)_{1 \leq \ell, m \leq p} $$
By the law of large numbers, this matrix a.s. tends to $\operatorname{Diag} (x_1, \dots, x_p)$.
Applying the continuous map $M \mapsto \operatorname{Tr} (e^{i \mu M/n})$ from the $p \times p$ matrices to 
$\mathbb{C}$, we deduce \eqref{convergencesp123}.

Let $A$ be strictly  larger than the maximum of $(|x_{\ell}|)_{1 \leq \ell \leq p}$. 
From \eqref{convergencesp1}, all the eigenvalues of $P_n/n$ are almost surely in $[-A,A]$ for $n$ large enough. Let $I$ be an interval whose boundary does not contain a point $x_{\ell}$. In order to prove  \eqref{convergencesp2}, it is enough 
to prove it for $I \cap [-A,A]$ instead of $I$, and then one can assume that $I$ is bounded: let $y \leq z$ be the endpoints of $I$. For $\epsilon > 0$ small enough, $[y - \epsilon, y + \epsilon]$ and $[z-\epsilon, z + \epsilon]$ do not contain any points of $(x_{\ell})_{\ell \geq 1}$, and then (by \eqref{convergencesp1}) no eigenvalue of $P_n/n$ for $n$ large enough. 
We deduce that if $f$ is a real-valued and continuous function from $\mathbb{R}$ to $\mathbb{R}$, equal to $1$ on $[y,z]$ and equal to zero on $[-A,A] \backslash [y - \epsilon, z + \epsilon]$,
it is enough to check 
\begin{equation} \sum_{\lambda \in \operatorname{Spec} (P_n) } n (\Pi_{P_n, \lambda})_{a,b} f(\lambda/n)  \, 
\underset{n \rightarrow \infty}{\longrightarrow} \sum_{\ell = 1}^p  \xi^{(\ell)}_{a}  \overline{\xi^{(\ell)}_{b}}  f(x_{\ell}),
\label{convergenceff}
\end{equation}
since $f$ coincides with the indicator of $I$ at all points $x_{\ell}$ and all eigenvalues of $P_n/n$ for $n$ large enough. 

In fact, we will prove \eqref{convergenceff} for all continuous functions $f$. Let us first assume that $f$ is a polynomial. On the subspace $E_n$ of $\mathbb{C}^n$, 
$$\sum_{\lambda \in \operatorname{Spec} (P_n) } n \Pi_{P_n, \lambda}f(\lambda/n)
= n f(P_n /n).$$
On the orthogonal of $E_n$, the same sum is equal to zero, since $P_n$ is only defined on $E_n$. Hence, 
$$ \sum_{\lambda \in \operatorname{Spec} (P_n) } n (\Pi_{P_n, \lambda})_{a,b} f(\lambda/n)
= n [f(P_n/n) \Pi_{E_n}]_{a,b},$$
where $\Pi_{E_n}$ is the orthogonal projection on $E_n$. 
We have seen that $( \xi_{[n]}^{(\ell)})_{1 \leq \ell \leq p}$ is a basis of $E_n$: let  $(v_{p+1}, v_{p+2}, \dots, v_n)$ be 
an orthonormal basis of $E_n^{\perp}$. These bases taken together give a basis $\mathcal{Q}$ of $\mathbb{C}^n$. We have previously computed $P_n/n$ in the basis $( \xi_{[n]}^{(\ell)})_{1 \leq \ell \leq p}$. From this computation, we 
deduce that the matrix of $ n f(P_n/n) \Pi_{E_n}$ in the basis $\mathcal{Q}$ is
$$R :=  \operatorname{Diag} \left( n f \left( \left( \frac{1}{n} \sum_{\ell = 1}^p x_{\ell} \langle \xi_{[n]}^{(\ell)} ,  \xi_{[n]}^{(m)} \rangle \right)_{1 \leq \ell, m \leq p} \right), ( 0)_{p+1 \leq \ell, m \leq n} \right)$$
Hence, if $Q$ is the matrix whose columns form the basis $\mathcal{Q}$, we get 
$$ \sum_{\lambda \in \operatorname{Spec} (P_n) } n (\Pi_{P_n, \lambda})_{a,b} f(\lambda/n)  
= (QRQ^{-1})_{a,b}. $$

We know the coefficients of $Q$. In order to estimate the coefficients of $Q^{-1}$, we consider the matrix $S$ obtained by dividing the $p$ first columns $( \xi_{[n]}^{(\ell)})_{1 \leq \ell \leq p}$ of $Q$ by $\sqrt{n}$. 
By the law of large numbers, for $1 \leq \ell, m \leq p$, the inner product of the columns $\ell$ and $m$ of $S$ tends to $\delta_{\ell,m}$ when $n$ goes to infinity. Moreover, the columns of index larger than $p$ are orthogonal to the $p$ first columns. Hence, if we apply Gram-Schmidt orthonormalization to the columns of $S$, we multiply 
$S$ at the right by a matrix of the form $\operatorname{Diag}(T, I_{n-p})$, where $T$ is a $p \times p$ matrix tending to identity when $n$ goes to infinity. We have: 
\begin{align*}
Q^{-1} & =  (S \operatorname{Diag}( \sqrt{n} I_p, I_{n-p}) )^{-1} \\ & = 
 \operatorname{Diag}( n^{-1/2}  I_p, I_{n-p}) \operatorname{Diag}(T, I_{n-p})
[S\operatorname{Diag}(T, I_{n-p})]^{-1}
\end{align*}
Since the product $S\operatorname{Diag}(T, I_{n-p})$ is a unitary matrix, 
we deduce 
$$Q^{-1} = \operatorname{Diag}(n^{-1/2} T T^*, I_{n-p}) S^*
= \operatorname{Diag}(n^{-1} T T^*, I_{n-p}) Q^*,
$$
and 
$$QRQ^{-1} = Q    \operatorname{Diag} \left( V, ( 0)_{p+1 \leq \ell, m \leq n} \right)
Q^*$$
where
 $$V = f \left( \left( \frac{1}{n} \sum_{\ell = 1}^p x_{\ell} \langle \xi_{[n]}^{(\ell)} ,  \xi_{[n]}^{(m)} \rangle \right)_{1 \leq \ell, m \leq p} \right) T T^*.$$
Now, $V$ tends to $\operatorname{Diag}(f(x_1), \dots, f(x_p))$ when $n$ goes to infinity.
We deduce that the entry $a,b$ of $QRQ^{-1}$ tends to the right-hand side of \eqref{convergenceff}, which shows this convergence when $f$ is a polynomial. 

In particular, taking $f = 1$ and $a = b$, we get 
$$ \sum_{\lambda \in \operatorname{Spec} (P_n) } n (\Pi_{P_n, \lambda})_{a,a}   \, 
\underset{n \rightarrow \infty}{\longrightarrow} \sum_{\ell = 1}^p |  \xi^{(\ell)}_{a} |^2,$$
which shows in particular that the left-hand side of this convergence is bounded independently of $n$.
Moreover, since $\Pi_{P_n, \lambda}$ is a positive operator, we have 
$$|(\Pi_{P_n, \lambda})_{a,b}| \leq [(\Pi_{P_n, \lambda})_{a,a}(\Pi_{P_n, \lambda})_{b,b}]^{1/2} \leq 
\frac{1}{2} \left(  (\Pi_{P_n, \lambda})_{a,a}  +  (\Pi_{P_n, \lambda})_{b,b} \right)$$
Hence, for all $a, b  \geq 1$, 
$$ \sum_{\lambda \in \operatorname{Spec} (P_n) } n |(\Pi_{P_n, \lambda})_{a,b} | \leq M_{a,b}$$
independently of $n$, for some random $M_{a,b} > 0$. If we choose 
$$M_{a,b} >  \sum_{\ell = 1}^p |  \xi^{(\ell)}_{a} |  |  \xi^{(\ell)}_{b} |,$$
we deduce that in \eqref{convergenceff}, for $n$ large enough in order to have all the spectrum of $P_n/n$ in $[-A,A]$, 
changing a function $f$ by a function $g$ changes the two sides by at most $M_{a,b} \sup_{[-A,A]} |f - g|$. 
Since any continuous function can be uniformly approached by polynomials on compact sets, we deduce that \eqref{convergenceff} extends to all  continuous functions. 
\section{Preliminary bound on the operator norm}
In this section, we will prove some bound on the limiting operator norm of a matrix satisfying  \eqref{formulaOV}. 
This bound is a consequence of the results in \cite{OV96} and \cite{BO01}, however, we give an alternative proof here for sake of completeness. 
\begin{proposition}
Let $M$ be an infinite matrix satisfying \eqref{formulaOV}. Then, almost surely, 
$$\underset{n \rightarrow \infty}{\lim \sup} \frac{||M_n||}{n} \leq   \left(\sum_{\ell \geq 1} x_{\ell}^2 \right)^{1/2},$$
where $||M_n||$ is the operator norm of the $n \times n$ top-left minor of $M$. 
\end{proposition} 
\begin{remark}
From Theorem \ref{convergenceeigenvectors}, the upper limit is in fact a limit and is equal to the maximum of $(|x_{\ell}|)_{\ell \geq 1}$.
\end{remark}
\begin{proof}
Shifting by a fixed multiple of identity does not change the upper limit. The Gaussian part is also irrelevant since 
the operator norm of the Gaussian Unitary Ensemble is a.s. negligible with respect to $n$ (it is classical that it is a.s. $\mathcal{O}(\sqrt{n})$, one can prove  that it is a.s. $\mathcal{O}(n^{5/6 + \epsilon})$,  just by expanding and bounding
$\mathbb{E} [\operatorname{Tr}(M_n^6)]$, and applying Borel-Cantelli lemma). We can then assume 
$$m_{j,k} = \sum_{\ell \geq 1} x_{\ell}  (\xi_j^{(\ell)} \overline{\xi_k^{(\ell)}} - \delta_{j,k} ). $$
We have 
\begin{align*}
||M_n||^2 \leq \operatorname{Tr} (M_n^2) 
& = \sum_{1 \leq j,k  \leq n} |m_{j,k}|^2
\\ & = \sum_{p, q \geq 1}  x_p x_q \sum_{1 \leq j, k \leq n}  (\overline{\xi_j^{(p)}} \xi_k^{(p)} - \delta_{j,k} )
 (\xi_j^{(q)} \overline{\xi_k^{(q)}} - \delta_{j,k} ).
\end{align*}
Here, the last sum can be infinite. In this case, it is rigorously defined as the a.s. limit of the sum on $1 \leq p, q \leq r$, 
when $r$ goes to infinity. 
Hence, 
$$||M_n||^2 \leq  \sum_{p, q \geq 1}  x_p x_q ( | \langle \xi_{[n]}^{(p)}, \xi_{[n]}^{(q)} \rangle |^2 - 
||\xi_{[n]}^{(p)}||^2 - ||\xi_{[n]}^{(q)}||^2 + n ),$$
$$||M_n||^2 \leq n^2  \sum_{p \geq 1} x_p^2 
+ \sum_{p, q \geq 1}  x_p x_q ( | \langle \xi_{[n]}^{(p)}, \xi_{[n]}^{(q)} \rangle |^2 - n^2 \delta_{p,q} -
||\xi_{[n]}^{(p)}||^2 - ||\xi_{[n]}^{(q)}||^2 + n ).$$
It is then sufficient to show that almost surely
$$ \sum_{p, q \geq 1}  x_p x_q ( | \langle \xi_{[n]}^{(p)}, \xi_{[n]}^{(q)} \rangle |^2 - n^2 \delta_{p,q} -
||\xi_{[n]}^{(p)}||^2 - ||\xi_{[n]}^{(q)}||^2 + n ) = o(n^2),$$
which is guaranteed, by Borel-Cantelli lemma, by the estimate
$$\mathbb{E} \left[ \left( \sum_{p, q \geq 1}  x_p x_q ( | \langle \xi_{[n]}^{(p)}, \xi_{[n]}^{(q)} \rangle |^2 - n^2 \delta_{p,q} -
||\xi_{[n]}^{(p)}||^2 - ||\xi_{[n]}^{(q)}||^2 + n ) \right)^4 \right] = \mathcal{O} (n^6),$$
and then (using Fatou's lemma in the case of an infinite sum), by 
$$\sum_{p_1, q_1, p_2, q_2, p_3, q_3, p_4, q_4 \geq 1} 
\prod_{s =1}^4 (x_{p_s} x_{q_s} ) \dots
$$ $$ \dots \times
\mathbb{E} \left[ \prod_{s = 1}^4   ( | \langle \xi_{[n]}^{(p_s)}, \xi_{[n]}^{(q_s)} \rangle |^2  - n^2 \delta_{p_s,q_s} -
||\xi_{[n]}^{(p_s)}||^2 - ||\xi_{[n]}^{(q_s)}||^2 + n ) \right]= \mathcal{O}(n^6)$$
where, in the case of infinite sums, we take a lower limit of the sums on $p_1, q_1, \dots, p_4, q_4 \leq r$ when $r$ goes to infinity. 
Let us now estimate the expectations in the last equation. If one of the eight indices $p_s, q_s$ appears 
exactly once (say $p_1$), we can first condition on all the seven other $\xi^{(p_s)}_{[n]}$,  $\xi^{(q_s)}_{[n]}$. In the conditional expectation, three of  the factors are fixed. The conditional expectation of the last factor is 
$$\mathbb{E} \left[ | \langle \xi_{[n]}^{(p_1)}, \xi_{[n]}^{(q_1)} \rangle |^2  - n^2 \delta_{p_1,q_1} -
||\xi_{[n]}^{(p_1)}||^2 - ||\xi_{[n]}^{(q_1)}||^2 + n \;  \big|  \xi_{[n]}^{(q_1)} \right].$$
Since $p_1 \neq q_1$ by assumption, we have
\begin{align*} 
\mathbb{E} \left[ | \langle \xi_{[n]}^{(p_1)} , \xi_{[n]}^{(q_1)} \rangle |^2 \;  \big|  \xi_{[n]}^{(q_1)} \right]
& = \sum_{1 \leq j, k \leq n}   \xi_{j}^{(q_1)}  \overline{\xi_{k}^{(q_1)} } 
\mathbb{E} [\overline{\xi_{j}^{(p_1)} } \xi_k^{(p_1)} ] \\ & = 
 \sum_{1 \leq j, k \leq n}   \xi_{j}^{(q_1)}  \overline{\xi_{k}^{(q_1)} }  \delta_{j,k} 
= ||\xi_{[n]}^{(q_1)}||^2,
\end{align*}
$$n^2 \delta_{p_1,q_1}  = 0,$$ 
$$\mathbb{E} \left[ ||\xi_{[n]}^{(p_1)}||^2 \;  \big|  \xi_{[n]}^{(q_1)} \right] = n,$$
$$\mathbb{E} \left[ ||\xi_{[n]}^{(q_1)}||^2 \;  \big|  \xi_{[n]}^{(q_1)} \right] = ||\xi_{[n]}^{(q_1)}||^2,$$
and then 
$$\mathbb{E} \left[ | \langle \xi_{[n]}^{(p_1)}, \xi_{[n]}^{(q_1)} \rangle |^2  - n^2 \delta_{p_1,q_1} -
||\xi_{[n]}^{(p_1)}||^2 - ||\xi_{[n]}^{(q_1)}||^2 + n \;  \big|  \xi_{[n]}^{(q_1)} \right] = 0.$$
We deduce that 
$$\mathbb{E} \left[ \prod_{s = 1}^4   ( | \langle \xi_{[n]}^{(p_s)}, \xi_{[n]}^{(q_s)} \rangle |^2  - n^2 \delta_{p_s,q_s} -
||\xi_{[n]}^{(p_s)}||^2 - ||\xi_{[n]}^{(q_s)}||^2 + n ) \right] = 0$$
as soon as one of the eight indices $p_s, q_2$ appears only once. 
Using H\"older inequality, it is then enough to show: 
$$\sum'_{p_1, q_1, p_2, q_2, p_3, q_3, p_4, q_4 \geq 1} 
\prod_{s =1}^4 |x_{p_s} x_{q_s} |\dots
$$ $$ \dots \times
\prod_{s = 1}^4  \mathbb{E}  \left[ \left( | \langle \xi_{[n]}^{(p_s)}, \xi_{[n]}^{(q_s)} \rangle |^2  - n^2 \delta_{p_s,q_s} -||\xi_{[n]}^{(p_s)}||^2 - ||\xi_{[n]}^{(q_s)}||^2 + n  \right)^4 \right] ^{1/4}= \mathcal{O}(n^6),$$
where the prime means that we restrict the sum to the terms where each of the indices appears at least twice.
It is then enough to show that 
\begin{equation}
\sum'_{p_1, q_1, p_2, q_2, p_3, q_3, p_4, q_4 \geq 1} 
\prod_{s =1}^4 |x_{p_s} x_{q_s} | < \infty \label{sumx}
\end{equation}
and 
\begin{equation}
 \mathbb{E}  \left[ \left( | \langle \xi_{[n]}^{(p)}, \xi_{[n]}^{(q)} \rangle |^2  - n^2 \delta_{p,q} -||\xi_{[n]}^{(p)}||^2 - ||\xi_{[n]}^{(q)}||^2 + n  \right)^4 \right] = \mathcal{O}(n^6).  \label{momentxi}
\end{equation}
For the first estimate \eqref{sumx}, we observe that in order to choose $p_1, q_1, p_2, q_2, p_3, q_3, p_4, q_4$, we 
have to choose: 
\begin{itemize}
\item The different indices which appear. 
\item The number of times each index appears. 
\item The exact positions where they appear: given the two first items, the number of possibilities is uniformly bounded. 
\end{itemize}
Hence, 
\begin{align*}
\sum'_{p_1, q_1, p_2, q_2, p_3, q_3, p_4, q_4 \geq 1}  &
\prod_{s =1}^4 |x_{p_s} x_{q_s} | \ll \sum_{p \geq 1} x_p^8 
+ \sum_{p \neq q \geq 1} ( x_p^6 x_q^2 + |x_p|^5 |x_q|^3 + x_p^4 x_q^4)  
\\ &
+ \sum_{p \neq q \neq r \geq 1} (x_p^4 x_q^2 x_r^2 + |x_p|^3 |x_q|^3 x_r^2)
 + \sum_{p \neq q \neq r \neq s \geq 1} x_p^2 x_q^2 x_r^2 x_s^2,
\end{align*}
which implies the crude bound
$$\sum'_{p_1, q_1, p_2, q_2, p_3, q_3, p_4, q_4 \geq 1}  
\prod_{s =1}^4 |x_{p_s} x_{q_s} | \ll \prod_{j= 2}^8 (1 + \sum_{p \geq 1} |x_p|^j). $$

Now, for $j \geq 2$, 
$$\sum_{p \geq 1} |x_p|^j \leq \max_{p \geq 1} |x_p|^{j-2} \sum_{p \geq 1} |x_p|^2 
\leq  \left( \sum_{p \geq 1} |x_p|^{2} \right)^{1 + (j-2)/2} < \infty,$$
which proves \eqref{sumx}. 
Let us now prove \eqref{momentxi}. We have, using H\"older inequality:
$$\mathbb{E} [ ||\xi_{[n]}^{(p)} ||^8 ] 
= \sum_{1 \leq a,b,c,d \leq n} \mathbb{E} [ |\xi_{a}^{(p)}|^2 |\xi_{b}^{(p)}|^2 |\xi_{c}^{(p)}|^2
|\xi_{d}^{(p)}|^2 ]  \leq \sum_{1 \leq a,b,c,d \leq n} \mathbb{E} [ |\xi_{1}^{(p)}|^8] = 24 n^4.$$
Hence, it is enough to show 
$$\mathbb{E}  \left[ \left( | \langle \xi_{[n]}^{(p)}, \xi_{[n]}^{(q)} \rangle |^2  - n^2 \delta_{p,q}  \right)^4 \right] = \mathcal{O}(n^6). $$
If $p \neq q$, the left-hand side is, by using the fact that $\xi_{[n]}^{(p)}, \xi_{[n]}^{(q)}$ are independent with the same distribution as $\xi_{[n]}^{(1)}$,
$$\mathbb{E} [  | \langle \xi_{[n]}^{(p)}, \xi_{[n]}^{(q)} \rangle |^8] 
= \sum_{1 \leq a_1, \dots, a_8 \leq n} \left|  \mathbb{E} \left[   \prod_{s=1}^4 \overline{\xi_{a_s}^{(1)}} 
\prod_{s=5}^8 \xi_{a_s}^{(1)} \right] \right|^2.$$
If one of the eight indices appears only once, the last expectation is zero by rotational invariance of the law of $\xi_{a_s}^{(1)}$. Hence, for all non-zero terms, there are at most four different indices among $a_1, \dots, a_8$. We deduce 
$$\mathbb{E} [  | \langle \xi_{[n]}^{(p)}, \xi_{[n]}^{(q)} \rangle |^8] = \mathcal{O}(n^4),$$
which is more than we need. 
For $p = q$, we have to show
$$\mathbb{E} [ ||\xi_{[n]}^{(p)}||^{16}] - 4 n^2 \mathbb{E} [ ||\xi_{[n]}^{(p)}||^{12}] + 6 n^4
\mathbb{E} [ ||\xi_{[n]}^{(p)}||^{8}] - 4 n^6 \mathbb{E} [ ||\xi_{[n]}^{(p)}||^{4}] + n^8 = \mathcal{O}(n^6).$$
For all $2 \leq r \leq 8$, 
$$\mathbb{E} [ ||\xi_{[n]}^{(p)}||^{2r}] 
= \sum_{1 \leq j_1, \dots, j_r \leq n} \mathbb{E} \left[ \prod_{s = 1}^r |\xi_{j_s}^{(p)}|^2 \right].$$
The sum of the terms where all the $j_s$ are distinct is equal to 
$$n (n-1) \dots (n-r+1) = n^r - \frac{r(r-1)}{2} n^{r-1} + \mathcal{O} (n^{r-2}).$$
The sum of the terms where two of the $j_s$ are equal and the others are all distinct is 
$$ \frac{r(r-1)}{2} n (n-1) \dots (n-r+2)  \mathbb{E} [ |\xi_{1}^{(p)}|^2]^{r-2} \mathbb{E} [ |\xi_{1}^{(p)}|^4]
= r(r-1) n^{r-1} + \mathcal{O}(n^{r-2}).$$
The sum of all the terms with more coincidences is $\mathcal{O}(n^{r-2})$. 
Hence, 
$$\mathbb{E} [ ||\xi_{[n]}^{(p)}||^{2r}] = n^r + \frac{r(r-1)}{2} n^{r-1} + \mathcal{O} (n^{r-2}),$$
and 
$$\sum_{s =0}^4 (-1)^s {4 \choose s} n^{8 - 2s} \mathbb{E} [ ||\xi_{[n]}^{(p)}||^{4s}]
= \sum_{s =0}^4 (-1)^s {4 \choose s} (n^8 - s (2s-1) n^7) + \mathcal{O}(n^6),$$
which is $\mathcal{O}(n^6)$, since
$$\sum_{s =0}^4 (-1)^s {4 \choose s} = 1 - 4 + 6 - 4 + 1 = 0,$$
and 
$$\sum_{s =0}^4 (-1)^s {4 \choose s} s(2s-1) = 1 \cdot 0 - 4 \cdot 1 + 6 \cdot 6 - 4 \cdot 15 + 1 \cdot 28
= 0 -4 + 36 - 60 + 28 = 0.$$
\end{proof}
\section{Proof of Theorem \ref{convergenceeigenvectors} in the general case}
In the convergence in distribution given by the theorem, it is enough to test intervals of the form
$(-\infty, c]$, $(-\infty, c)$, for $c < 0$ not in the sequence $(x_{\ell})_{\ell \geq 1}$, and the intervals of the form $[c, \infty)$, $(c, \infty)$ for $c > 0$ not in the sequence $(x_{\ell})_{\ell \geq 1}$. By symmetry, we only consider positive intervals, we fix $c$ and we denote by $I$ one of the intervals $[c, \infty)$ and $(c, \infty)$.
We define $\delta > 0$ as the minimum distance between $c$ and a point of $\{0,(x_{\ell})_{\ell \geq 1}\}$. 
For $\epsilon \in (0,c)$, we can decompose the infinite matrix $M$ as $A + B$, where 
$A = (a_{j,k})_{j, k \geq 1}$, $B = (b_{j,k})_{j,k \geq 1}$, 
$$a_{j,k} = \sum_{|x_{\ell}| > \epsilon} ( x_{\ell} \xi_{j}^{(\ell)} \overline{ \xi_{k}^{(\ell)} } - \delta_{j,k}),$$
$$b_{j,k} = \gamma_1 \delta_{j,k} + \sqrt{\gamma_2} G_{j,k} + \sum_{|x_{\ell}| \leq \epsilon}  ( x_{\ell} \xi_{j}^{(\ell)} \overline{ \xi_{k}^{(\ell)} } - \delta_{j,k}).$$
If $A_n$ and $B_n$ are the top-left $n \times n$ minors of $A$ and $B$, the preliminary bound shows that 
almost surely, for $n$ large enough, 
$$\frac{||B_n||}{n} \leq  \left[ \left(\sum_{|x_{\ell}| \leq \epsilon} x_{\ell}^2 \right)^{1/2} + \epsilon \right],$$
which, by dominated convergence, tends to zero with $\epsilon$. Hence, if we take $\epsilon \in (0,c)$ small enough, we can assume that almost surely, $||B_n|| \leq n \delta/2$ for $n$ large enough. 

We have first to show that the number of eigenvalues of $M_n/n$ which are strictly larger than $c$ (resp. larger than or equal to $c$) tends to the number $r$ of points of $(x_{\ell})_{\ell \geq 1}$ which are larger than $c$ (resp. larger than or equal to $c$). Now, from the finite case studied before, the $r$-th largest eigenvalue of $A_n/n$ tends to the $r$-th largest point 
of $(x_{\ell})_{\ell \geq 1}$,  which is at least $c+ \delta$ (recall that there is no point in the sequence in $(c-\delta, c+ \delta)$ by definition of $\delta$). Since $||B_n|| \leq n \delta/2$ for $n$ large, the lower limit of the $r$-th largest eigenvalue of $M_n/n$ is at least $c + \delta/2$. Similarly, the upper limit of the $(r+1)$-th largest eigenvalue of $M_n/n$ is 
at most $c - \delta/2$. Hence, for $n$ large enough, there are exactly $r$ eigenvalues of $M_n/n$ which are strictly larger than $c$ (resp. larger than or equal to $c$). 

It now remains to show that 
$$n (\Pi_{M_n, n I})_{a,b} \underset{n \rightarrow \infty}{\longrightarrow} \sum_{x_{\ell} \in I} \xi_a^{(\ell)} \overline{\xi_{b}^{(\ell)}},$$
where $\Pi_{M_n, n I}$ is the projection on the space $\mathcal{E}$ generated by the eigenvectors of $M_n$ associated to eigenvalues in $nI$ (recall that $I = [c, \infty)$ or $I = (c, \infty)$). 
By the study of the finite case, it is known that with obvious notation, 
$$n (\Pi_{A_n, n I})_{a,b} \underset{n \rightarrow \infty}{\longrightarrow} \sum_{x_{\ell} \in I} \xi_a^{(\ell)} \overline{\xi_{b}^{(\ell)}}.$$
Hence, if $\mathcal{F}$ is the space generated by the eigenvectors of $A_n$  associated to eigenvalues in $nI$, it is enough 
to show that 
$$n | (\Pi_{\mathcal{E}})_{a,b} -  (\Pi_{\mathcal{F}})_{a,b}|  \underset{n \rightarrow \infty}{\longrightarrow}  0.$$
Let $v$ be a unit eigenvector, corresponding to an eigenvalue $\lambda \in nI$ of $M_n$. We have a decomposition 
$v = w + x$, where $w \in \mathcal{F}$ and $x$ is orthogonal to $\mathcal{F}$. 
We have
$$(A_n + B_n) (w + x) = \lambda (w + x)$$
and then, taking the inner product with $x$: 
$$\langle x , A_n w \rangle + \langle x , B_n w \rangle + 
\langle x , A_n x \rangle + \langle x , B_n x \rangle = \lambda( \langle x, w \rangle + \langle x, x \rangle )
= \lambda ||x||^2 \geq nc ||x||^2.$$
The space $\mathcal{F}$ is stable by $A_n$, so 
$$\langle x , A_n w \rangle  = 0.$$
Since $x$ is orthogonal to $\mathcal{F}$, we have 
$$\langle x , A_n x \rangle \leq ||x||^2 \lambda_{n,(nI)^c},$$
where $\lambda_{n,(nI)^c}$ is the largest eigenvalue of $A_n$ in the complement of $nI$. 
Now, by the eigenvalue convergence in the finite case,  $n^{-1} \lambda_{n,(nI)^c}$ tends to the largest point of 
 $\{0, (x_{\ell})_{\ell \geq 1, 
x_{\ell} > \epsilon}\} \cap [0,c]$, and then it is smaller than $c - 3\delta/4$ for $n$ large enough, independently of the choice of the vector $v$. Hence, 
$$\langle x , A_n x \rangle \leq n (c - 3\delta/4) ||x||^2.$$
Moreover, for $n$ large enough (independently of $v$), 
$$\langle x, B_n x \rangle \leq  ||B_n|| \, ||x||^2 \leq n( \delta/2) ||x||^2.$$
Hence, we have 
$$\langle x , B_n w \rangle +  n (c - 3\delta/4) ||x||^2 + n (\delta/2) ||x||^2 \geq nc ||x||^2,$$
and
$$n \delta ||x||^2 / 4 \leq  |\langle x , B_n w \rangle| \leq ||x|| \, ||B_n w||,$$
$$||x|| \leq \frac{4}{n \delta} ||B_n w||.$$
Let $\mathcal{B} = \{y_1, \dots y_s\}$ be an orthonormal basis of $\mathcal{F}$, chosen as a measurable function of $A_n$. If $s < r$, we arbitrarily define $y_{s+1}= y_{s+2} = \dots = y_r = y_s$. 
Since $||w|| \leq ||v|| = 1$, by decomposing $w$ in the basis $\mathcal{B}$ and applying triangle inequality: 
$$||x|| \leq \frac{4}{n \delta} \sum_{j=1}^s  ||B_n y_j||. $$
Since the number of eigenvalues of $A_n$ in $nI$ is almost surely equal to $r$ for $n$ large enough, we get 
$$||x|| \leq \frac{4}{n \delta} \sum_{j=1}^r  ||B_n y_j||$$
for $n$ large enough. 
Let $U_j$ be a unitary matrix, chosen as a measurable function of $A_n$, such that $U_j y_j$ is  the first basis vector 
 $e$ of $\mathbb{C}^n$. By construction, $A_n$ and $B_n$ are independent. Since the law of $B_n$ is invariant by conjugation, $B_n$ has the same law as $U_j^{-1} B_n U_j$, conditionally on $A_n$, and $(B_n,y_j)$ has the same law as 
$(U_j^{-1} B_n U_j, y_j)$. Hence, we have the equality in distribution: 
$$||B_n y_j|| \overset{d}{=} ||U_j^{-1} B_n U_j y_j|| = ||B_n e||$$
Now, let us estimate the tail of the distribution of $||B_n e||$, by looking at its fourth moment.
We have 
$$\mathbb{E} [ ||B_n e||^4] \ll \mathbb{E} [ ||\gamma_1 e||^4] 
+  \mathbb{E} [ ||G_n e||^4] +  \mathbb{E} [ ||C_n e||^4],$$
where, with obvious notation, 
$$g_{j,k} = \sqrt{\gamma_2} G_{j,k},$$
$$c_{j,k} = \sum_{|x_{\ell}| \leq \epsilon} (x_{\ell} \xi_j^{(\ell)} \overline{\xi_k^{(\ell)} } - \delta_{j,k} ).$$
The first term is independent of $n$ (equal to $\gamma_1^4$). The second term is the $L^4$ norm of a Gaussian vector, easily dominated by $n^2$. The third term is 
$$\mathbb{E} \left[ \left( \sum_{j = 1}^n \left| \sum_{|x_{\ell}| \leq \epsilon} x_{\ell} ( \xi_j^{(\ell)}\overline{\xi_1^{(\ell)} } - \delta_{j,1} )
\right|^2 \right)^2 \right]$$
Using Fatou's lemma in the case where $(x_{\ell})_{\ell \geq 1}$ is infinite, we get that this expectation is bounded by 
$$ \sum_{|x_\ell|, |x_m|, |x_p|, |x_q| \leq \epsilon}  x_{\ell} x_m x_p x_q \sum_{1 \leq j, k \leq n}
\mathbb{E} \left[ ( \xi_j^{(\ell)}\overline{\xi_1^{(\ell)} } - \delta_{j,1} ) ( \xi_1^{(m)}\overline{\xi_j^{(m)} } - \delta_{j,1} ) \right. $$ $$ \left. \dots 
( \xi_k^{(p)}\overline{\xi_1^{(p)} } - \delta_{k,1} ) ( \xi_1^{(q)}\overline{\xi_k^{(q)} } - \delta_{k,1} ) \right],$$
where in the case of an infinite sequence, we restrict the sum to $1 \leq \ell, m, p, q \leq t$ and let $t \rightarrow \infty$. 
If one of the indices $\ell, m, p, q$, say $\ell$, is different from the three others,
we can use independence in the last expectation, in order to get a factor $\mathbb{E} [ \xi_{j}^{(\ell)} \overline{\xi_1^{(\ell)}}  - \delta_{j,1}] = 0 $. Hence, the only non-zero terms in the sum correspond to the case where the indices  $\ell, m, p, q$ are pairwise equal. Since the last expectation is uniformly bounded in any case, we deduce that the 
sum is bounded by a universal constant times $n^2 \left(\sum_{|x_{\ell}| \leq \epsilon} x_{\ell}^2 \right)^2.$
We have then proven 
$$\mathbb{E} [ ||B_n y_j||^4] =  \mathbb{E} [  ||B_n e||^4] \ll n^2,$$
and by Borel-Cantelli lemma, for all $j \in \{1, \dots, r\}$ and for all $n$ large enough, 
 $||B_n y_j|| \leq (\delta/(1 + 4r)) n^{4/5}$. 
Hence, almost surely, $||x|| \leq n^{-1/5}$ for $n$ large enough, uniformly on the choice of the unit  eigenvector $v$ in $\mathcal{E}$. In other words, almost surely, for $n$ large enough,  all unit eigenvectors of $M_n$ in $\mathcal{E}$ are at distance at most $n^{-1/5}$ from a vector in $\mathcal{F}$. 
For $n$ large, $\mathcal{E}$ and $\mathcal{F}$ have dimension $r$. If the eigenvectors $v_1, \dots, v_r$ of $M_n$ form an orthonormal basis of $\mathcal{E}$, we have vectors in $\mathcal{F}$ of the form $v_j + \mathcal{O}(n^{-1/5})$. 
The inner product of $v_j + \mathcal{O}(n^{-1/5})$ with $v_k + \mathcal{O}(n^{-1/5})$ is $\delta_{j,k}  + \mathcal{O}(n^{-1/5})$.
Applying Gram-Schmidt orthogonalization, we deduce that for $n$ large, one gets an orthonormal  basis of $\mathcal{F}$ of the form $(v_j + \mathcal{O}(n^{-1/5}))_{1 \leq j \leq r}$. Hence, for any unit vector $x$, and $n$ large, 
$$\Pi_{\mathcal{E}}(x) - \Pi_{\mathcal{F}} (x) 
= \sum_{j=1}^r  \langle v_j , x \rangle v_j  - \sum_{j=1}^r  \langle (v_j + \mathcal{O}(n^{-1/5}))  , x \rangle 
(v_j + \mathcal{O}(n^{-1/5})),$$
which implies that the operator norm of $\Pi_{\mathcal{E}} - \Pi_{\mathcal{F}}$ are a.s. dominated by $n^{-1/5}$.
On the other hand, since $A_n$ and $B_n$ are independent and unitarily invariant, the couple 
$(A_n, B_n)$ has the same law as $(U A_n U^{-1}, U B_n U^{-1})$, for all deterministic $U \in U(n)$. Now, simultaneous conjugation of $A_n$ and $B_n$ changes the spaces $(\mathcal{E},\mathcal{F})$ to their images by $U$. Hence, 
$(\mathcal{E},\mathcal{F})$  has the same law as $(U \mathcal{E}, U\mathcal{F})$, and 
$\Pi_{\mathcal{E}} - \Pi_{\mathcal{F}}$ is invariant by unitary conjugation. We deduce that there is an equality in distribution of the form 
$$\Pi_{\mathcal{E}} - \Pi_{\mathcal{F}} \overset{d}{=} U \Lambda U^{-1},$$
where $\Lambda$ is a random diagonal matrix whose entries have nonincreasing modulus, and $U = (u_{j,k})_{1 \leq j, k \leq n}$ is an independent, Haar-distributed matrix in $U(n)$. 
Since $\mathcal{E}$ and $\mathcal{F}$ have dimension $r$, $\Pi_{\mathcal{E}} - \Pi_{\mathcal{F}}$ has at most rank $2r$, and then only the $2r$ first entries of $\Lambda$ can be different from zero. 
We get (for $n$ large) 
$$|(\Pi_{\mathcal{E}} - \Pi_{\mathcal{F}})_{a,b}| = \left| \sum_{j = 1}^{2r} \Lambda_j u_{a,j} \overline{u_{b,j}}
\right| \leq 2r ||\Pi_{\mathcal{E}} - \Pi_{\mathcal{F}}|| \sup_{1 \leq j,k \leq n} |u_{j,k}|^2.$$ 
Now, $|u_{j,k}|^2$ is a Beta variable of parameters $1$ and $n-1$, which implies 
$$\mathbb{P} [ |u_{j,k}|^2 \geq n^{-0.99} ] 
= (n-1) \int_{n^{-0.99}}^1 (1-x)^{n-2} dx = (1 - n^{-0.99} )^{n-1} \ll e^{- n^{0.01} }.$$
Using Borel-Cantelli lemma, and the previous estimate on $||\Pi_{\mathcal{E}} - \Pi_{\mathcal{F}}|| $, we deduce that almost surely, 
$$|(\Pi_{\mathcal{E}} - \Pi_{\mathcal{F}})_{a,b}|  = \mathcal{O} (n^{-1.19}),$$
which completes the proof of Theorem \ref{convergenceeigenvectors}.

\bibliographystyle{halpha}

\end{document}